\documentclass[12pt,a4paper]{article}

\usepackage{graphicx}
\usepackage{amsmath}
\usepackage{amsthm}
\usepackage{amssymb}
\usepackage{amsfonts}
\usepackage{latexsym}
\usepackage{epsfig}
\usepackage{color}

\newcommand{\beql}[1]{\begin{equation}\label{#1}}
\newcommand{\eeq}{\end{equation}}
\newcommand{\comment}[1]{}

\newcommand{\RR}{{\mathbb R}}

\newcommand{\ZZ}{{\mathbb Z}}
\newcommand{\NN}{{\mathbb N}}

\newcommand{\QQ}{{\mathbb Q}}
\newcommand{\BB}{{\mathcal B}}
\newcommand{\SA}{{\mathcal S}}
\newcommand{\D}{{\Delta}}

\newcommand{\ve}{{\varepsilon}}

\newcounter{open}
\newcounter{dfn}
\def\thedfn{\arabic{dfn}}

\newcounter{obs}
\def\theobs{\arabic{obs}}

\newcounter{thm}
\newcounter{othm}
\def\theothm{\Alph{othm}}

\newcounter{mysec}
\newcounter{mysubsec}[mysec]
\newtheorem{theorem}{Theorem}
\newtheorem{corollary}{Corollary}
\newtheorem{lemma}{Lemma}
\newtheorem{proposition}{Proposition}
\newtheorem{definition}{Definition}

\newtheorem{problem}{Problem}
\theoremstyle{definition}
\newtheorem{example}{Example}
\newtheorem{remark}{Remark}

\newcounter{rem}
\newcounter{rev}
\begin{document}

\title{On uniform asymptotic upper density in locally compact abelian
groups}

\author{Szil\' ard Gy. R\' ev\' esz{\thanks{Supported in part by
the Hungarian National Foundation for Scientific Research, Project
\#s T-49301, K-61908 and K-72731.}}}

\let\oldfootnote\thefootnote
\def\thefootnote{}

\maketitle

{\small \tableofcontents}


\begin{abstract}
Starting out from results known for the most classical cases of
$\NN$, $\ZZ^d$, $\RR^d$ or for $\sigma$-finite abelian groups,
here we define the notion of asymptotic uniform upper density in
general locally compact abelian groups. Even if a bit surprising,
the new notion proves to be the right extension of the classical
cases of $\ZZ^d$, $\RR^d$. The new notion is used to extend some
analogous results previously obtained only for classical cases or
$\sigma$-finite abelian groups. In particular, we show the
following extension of a well-known result for $\ZZ$ of
F\"urstenberg: if in a general locally compact Abelian group $G$ a
set $S\subset G$ has positive uniform asymptotic upper density,
then $S-S$ is syndetic.
\end{abstract}



\newpage

\section{Measuring large, but not necessarily dense infinite
sequences and sets in groups}

Our aim here is to extend the notion of \emph{uniform asymptotic
upper density}, used in case of $\RR$ already by Beurling and
P\'olya in the analysis of entire functions. The same notion is
frequently called by others as \emph{Banach density}, c.f. e.g.
\cite[p. 72]{Füsi}.

The notion of uniform asymptotic upper density -- u.a.u.d. for
short -- is a way to grab the idea of a set being relatively
considerable, even if not necessarily dense or large in some other
more easily accessible sense. In many theorems, in particular in
Fourier analysis and in additive problems where difference sets or
sumsets are considered, the u.a.u.d. is the right notion to
express that a set becomes relevant in the question considered.
However, to date the notions was only extended to sequences and
subsets of the real line, and some immediate relatives like
$\ZZ^d$, $\RR^d$, as well as to finite, or at least finitely
constructed (e.g. $\sigma$-finite) cases.

A framework where the notion might be needed is the generality of
LCA groups. In recent decades it is more and more realized that
many questions e.g. in additive number theory can be investigated,
even sometimes structurally better understood/described, if we
leave e.g. $\ZZ$, and consider the analogous questions in Abelian
groups. In fact, when some analysis, i.e. topology also has a role
-- like in questions of Fourier analysis e.g. -- then the setting
of LCA groups seems to be the natural framework. And inded several
notions and questions, where in classical results u.a.u.d. played
a role, have already been defined, even in some extent discussed
in LCA groups. Nevertheless, it seems that no attempt has been
made to extend the very notion of u.a.u.d. to this setup.

One of the more explicit attempts to really "measure sets in
infinite groups" is perhaps the work of Borovik at al.
\cite{borovikatal:freegroups}, \cite{borovikatal:infinitegroups}.
Other papers, where some ideas close to ours can be seen, are
\cite{Mört} -- considering measures, not sets, although the
investigation there is focused on local structure at small
neighborhoods of points -- and in \cite{nguetseng:meanvalue},
where at least the setup of LCA groups is apparent (although the
interest is quite different).

For cases of $\sigma$-finite groups $G$ it is easy to design the
u.a.u.d., compare \cite{hegyvari:differences}. In the more general
framework of discrete groups, I.Z. Ruzsa \cite{ruzsa:oral} had two
constructions to define u.a.u.d..

However, neither of these constructions were the same as ours.
Below we will explain, how one may construct notions of u.a.u.d.,
which finely extend the classical notion.

\section{Some additive number theory flavored results for difference sets}
\label{sec:additiveresults}

\comment{Stewart and Tijdeman \cite{stewart-tijdeman} observed
that if a sequence $A\subset\NN$ has positive upper density,
meaning that $\overline{d}(A):=\lim\sup_{n\to\infty} A(n)/n > 0$
with $A(n):= {\rm \#} (A\cap [1,n])$, then the iterated difference
sequence $D_k(A)$, defined as $D_0(A)=A$ and
$D_{k+1}(A):=D_k(A)-D_k(A)$ for $k\in\NN$, stabilizes after a
finite number of iterations. The best upper bound for the time
(i.e., the first index $k_0$) of stabilization was estimated by
Tijdeman-Stewart, loc. cit., and improved by Ruzsa
\cite{ruzsa:iterated}.}

Let us denote the upper density of $A\subset \NN$ as
$\overline{d}(A):=\lim\sup_{n\to\infty} A(n)/n > 0$ with $A(n):=
{\rm \#} (A\cap [1,n])$. Erd\H os and S\'ark\"ozy (seemingly
unpublished, but quoted in \cite{hegyvari:differences} and in
\cite{ruzsa:difference}) observed the following.

\begin{proposition}[Erd\H os-S\'ark\"ozi]\label{prop:ErdSar}
If the upper density $\overline{d}(A)$ of a sequence $A\subset
\NN$ is positive, then writing the positive elements of the
sequence $D(A):=D_1(A):=A-A$ as $D(A) \cap \NN
=\{(0<)d_1<d_2<\dots\}$ we have $d_{n+1}-d_n=O(1)$.
\end{proposition}
This is analogous, but not contained in the following result of
Hegyv\'ari, obtained for $\sigma$-finite groups. An abelian group
is called $\sigma$-finite (with respect to $H_n$), if there exists
an increasing sequence of {\em finite} subgroups $H_n$ so that
$G=\cup_{n=1}^\infty H_n$. For such a group Hegyv\'ari defines
asymptotic upper density (with respect to $H_n$) of a subset $A
\subset G$ as
\begin{equation}\label{GHdensity}
\overline{d}_{H_n}(A) := \limsup_{n\to\infty} \frac{{\rm \#}
(A\cap H_n)}{{\rm \#} H_n}~.
\end{equation}
Note that for finite groups this is just ${\rm \#} (A\cap G) /
{\rm \#} G$. Hegyv\'ari proves the following \cite[Proposition
1]{hegyvari:differences}.

\begin{proposition}[Hegyv\'ari]\label{thm:hegyvari}
Let $G$ be a $\sigma$-finite abelian group with respect to the
increasing, exhausting sequence $H_n$ of finite subgroups and let
$A\subset G$ have positive upper density with respect to $H_n$.
Then there exists a finite subset $B\subset G$ so that $A-A+B=G$.
Moreover, we have ${\rm \#}B\le 1/\overline{d}_{H_n}(A)$.
\end{proposition}

F\"urstenberg calls a subset $S\subset G$ in a topological Abelian
(semi)group a {\em syndetic} set, if there exists a compact set
$K\subset G$ such that for each element $g\in G$ there exists a
$k\in K$ with $gk\in S$; in other words, in topological groups
$\cup_{k\in K} Sk^{-1}=G$. Then he presents as Proposition 3.19
(a) of \cite{Füsi} the following.

\begin{proposition}[F\"urstenberg]\label{prop:Füst} Let $S\subset
\ZZ$ with positive upper Banach density. Then $S-S$ is a syndetic
set.
\end{proposition}

In the following we extend the notion of uniform asymptotic upper
density, (also called as Banach density) to arbitrary LCA groups,
and present various generalized versions of the above results,
which cover all of them.

In fact, our interest in the problem of the definition of u.a.u.d.
in general LCA groups came from another problem, the so-called
Tur\'an extremal problem for positive definite functions. In that
question some results, already known for classical situations like
$\RR^d$, $\ZZ^d$ or compact groups, can also be extended. We
discuss these questions in \cite{R}.

\section{Various forms of the asymptotic density}

\noindent We start with the frequently used definition of
asymptotic upper density in $\RR^d$. Let $K\subset\RR^d$ be a
\emph{fat body}, i.e. a set with $0\in {\rm int} K$,
$K=\overline{\rm int K}$ and $K$ compact. Then asymptotic upper
density with respect to $K$ is defined as
\begin{equation}\label{Rdensity}
\overline{d}_{K}(A) := \limsup_{r\to\infty} \frac{|A\cap
rK|}{|rK|}~.
\end{equation}
The definition \eqref{GHdensity} is clearly analogous to
\eqref{Rdensity}. As is easy to see, both \eqref{Rdensity} and
\eqref{GHdensity} depends on the choice of the fundamental set $K$
or sequence $H_n$, even if {\em positivity} of \eqref{Rdensity} is
invariant for a large class of underlying sets including all
convex, but also many other bodies. The similar notion of density
applies and has the same properties also for the discrete group
$\ZZ^d$. On the other hand, for a given subset $A$ in a
$\sigma$-finite group $G$, \eqref{GHdensity} can easily be zero
for some fundamental sequence $H_n$, while being maximal (i.e., 1)
for some other choice $H'_n$ of fundamental sequence.

\begin{example}\label{ex:1}
Let $G:=\QQ/\ZZ$, which is a $\sigma$-finite additive abelian
group. Let $H_n:=\{r\in G~:~ r=\frac pq ,~ q\le n \}$; then $H_n$
is an increasing and exhausting sequence of finite subgroups of
$G$. Note that ${\rm \#}H_n=\sum_{j\le n}\varphi(j)\sim
\frac{6}{\pi^2} n^2$. Let then $A_k:= \{r\in G~:~ r=\frac pq ,~
(p,q)=1,~ (k^2+k)!< q \le (k+1)^2! \}$ and $A:=\cup_{k=1}^\infty
A_k$. Then it is not hard to prove that $\liminf_{n\to\infty}
\frac{{\rm \#} A\cap H_n}{{\rm \#} H_n}=0$ but
$\limsup_{n\to\infty} \frac{{\rm \#} A\cap H_n}{{\rm \#} H_n}=1$.
Then it is clear that the value of the upper density can be either
0 or 1 depending on the choice of an appropriate subsequence of
$H_n$ as fundamental sequence. With a little modification an
example with arbitrary numbers as possible upper densities can be
derived.
\end{example}

However, results corresponding to the above ones of Erd\H os,
S\'ark\"ozi and Hegyv\'ari are easily sharpened by using only a
weaker notion, that of \emph{asymptotic uniform upper density}. It
could be defined as

\begin{equation}\label{GHUdensity}
\overline{D}_{H_n}(A) := \limsup_{n\to\infty} \frac{\sup_{x\in G}
{\rm \#} A\cap (H_n+x)}{{\rm \#} H_n}
\end{equation}
for $\sigma$-finite abelian groups and is defined as
\begin{equation}\label{RUdensity}
\overline{D}_{K}(A) := \limsup_{r\to\infty} \frac{\sup_{x\in\RR^d}
|A\cap (rK+x)|}{|rK|}~.
\end{equation}
in $\RR^d$. It is obvious that these notions are translation
invariant, and  $\overline{D}_{H_n}(A) \ge \overline{d}_{H_n}(A)$,
$\overline{D}_{K}(A) \ge \overline{d}_{K}(A)$. It is also
well-known, that $\overline{D}_{K}(A)$ gives the same value for
all nice - e.g. for all convex - bodies $K\subset \RR^d$, although
this fact does not seem immediate from the formulation. Actually,
we will obtain this as a side result, being an immediate corollary
of Theorem \ref{th:Requivalence}, see Remark \ref{c:easy}.

Similar definitions can be used for $\ZZ^d$. However, dependence
on the fundamental sequence $H_n$ makes the $\sigma$-finite case
less appealing, and we lack a successful notion for abelian groups
in general. In particular, a natural requirement is to find a
common generalization of asymptotic upper density, which works
both for $\RR^d$ and $\ZZ^d$, and also for a larger class of (say,
abelian) groups, including, but not restricted to $\sigma$-finite
ones.

Note also the following ambiguity in the use of densities in
literature. Sometimes even in continuous groups a discrete set
$\Lambda$ is considered in place of $A$, and then the definition
of the asymptotic upper density is
\begin{equation}\label{RUNdensity}
\overline{D}^{\#}_{K}(A) := \limsup_{r\to\infty}
\frac{\sup_{x\in\RR^d} {\rm \#} \Lambda \cap (rK+x)}{|rK|}~.
\end{equation}

That motivates our further extension: we are aiming at asymptotic
uniform upper densities of {\em measures}, say measure $\nu$ with
respect to measure $\mu$, (whether related by $\nu$ being the trace of $\mu$ on a set or not). 
E.g. in \eqref{RUNdensity} $\nu:={\rm \#}$ is the cardinality or
counting measure of a set $\Lambda$, while $\mu:=|\cdot|$ is just
the volume. The general formulation in $\RR^d$ is thus
\begin{equation}\label{Rnu-density}
\overline{D}_{K}(\nu):= \limsup_{r\to\infty}
\frac{\sup_{x\in\RR^d}\nu(rK+x)}{|rK|}~.
\end{equation}

Of course, to extend these notions some natural hypotheses should
apply. We are considering abelian groups (although non-abelian
groups come to mind naturally, here we do not consider this
extension), and in accordance to the group settings only densities
with respect to translation-invariant measures $\mu$ are suitable.
Otherwise we want $\nu$ to be a measure, possibly infinite, and
$\mu$ be another, translation-invariant, nonnegative (outer)
measure with strictly positive, but finite values when applied to
sets considered.

We will consider two generalizations here. The first applies for
the class of abelian groups $G$, equipped with a topological
structure which makes $G$ a LCA (locally compact abelian) group.
Considering such groups are natural for they have an essentially
unique translation invariant Haar measure $\mu_G$ (see e.g.
\cite{rudin:groups}), what we fix to be our $\mu$. By
construction, $\mu$ is a Borel measure, and the sigma algebra of
$\mu$-measurable sets is just the sigma algebra of Borel mesurable
sets, denoted by $\BB$ throughout. Furthermore, we will take
$\BB_0$ to be the members of $\BB$ with compact closure: note that
such Borel measurable sets necessarily have finite Haar measure.
This will be important for not allowing a certain degeneration of
the notion: e.g. if we consider $G=\RR$, $\nu$ is the counting
measure $\#$ and $A$ is some sequence $A=\{a_k~:~k\in \NN\}$, say
tending to infinity, then it is easy to define a (non-compact, but
still measurable) union $V$ of decreasingly small neighborhoods of
the points $a_k$ such that the Haar measure of $V$ does not exceed
1, but all of $A$ stays in $V$, hence the relative density of $A$,
with respect to the counting measure, is infinite. (Another way to
deal with this phenomenon would have been to fix that
$\infty/\infty=0$, but we prefer not to go into such questions.)

Note if we consider the discrete topological structure on any
abelian group $G$, it makes $G$ a LCA group with Haar measure
$\mu_G={\rm \#}$, the counting measure. Therefore, our notions
below certainly cover all discrete groups. This is the natural
structure for $\ZZ^d$, e.g. On the other hand all $\sigma$-finite
groups admit the same structure as well, unifying considerations.
(Note that $\ZZ^d$ is not a $\sigma$-finite group since it is {\em
torsion-free}, i.e. has no finite subgroups.)

The other measure $\nu$ can be defined, e.g., as the {\em trace}
of $\mu$ on the given set $A$, that is,
$\nu(H):=\nu_A(H):=\mu_G(H\cap A)$, or can be taken as the
counting measure of the points included in some set $\Lambda$
derived from the cardinality measure similarly:
$\gamma(H):=\gamma_{\Lambda}(H):={\rm \#} (H\cap \Lambda)$.

\begin{definition}\label{compactdensity}
Let $G$ be a LCA group and $\mu:=\mu_G$ be its Haar measure. If
$\nu$ is another measure on $G$ with the sigma algebra of
measurable sets being ${\mathcal S}$, then we define
\begin{equation}\label{Cnudensity}
\overline{D}(\nu;\mu) := \inf_{C\Subset G} \sup_{V\in {\mathcal S}
\cap \BB_0} \frac{\nu(V)}{\mu(C+V)}~.
\end{equation}
In particular, if $A\subset G$ is Borel measurable and $\nu=\mu_A$
is the trace of the Haar measure on the set $A$, then we get
\begin{equation}\label{CAdensity}
\overline{D}(A) :=\overline{D}(\nu_A;\mu) := \inf_{C\Subset G}
\sup_{V\in {\mathcal B}_0} \frac{\mu(A\cap V)}{\mu(C+V)}~.
\end{equation}
If $\Lambda\subset G$ is any (e.g. discrete) set and $\gamma
:=\gamma_\Lambda:=\sum_{\lambda\in\Lambda} \delta_{\lambda}$ is
the counting measure of $\Lambda$, then we get
\begin{equation}\label{CLdensity}
\overline{D}^{\#}(\Lambda) :=\overline{D}(\gamma_{\Lambda};\mu) :=
\inf_{C\Subset G} \sup_{V\in \BB_0} \frac{{\rm \# }(\Lambda\cap
V)}{\mu(C+V)}~.
\end{equation}
\end{definition}

\begin{theorem}\label{th:Requivalence}
Let $K$ be any convex
body in $\RR^d$ and normalize the Haar measure of $\RR^d$ to be
equal to the volume $|\cdot|$. Let $\nu$ be any measure with sigma
algebra of measurable sets $\mathcal S$. Then we have
\begin{equation}\label{Rd-equivalance}
\overline{D}(\nu;|\cdot|) = \overline{D}_K(\nu)~.
\end{equation}
The same statement applies also to $\ZZ^d$.
\end{theorem}

\begin{remark}\label{c:easy}
In particular, we find that the asymptotic uniform upper density
$\overline{D}_K(\nu)$ does not depend on the choice of $K$. For a
direct proof of this one has to cover the boundary of a large
homothetic copy of $K$ by standard (unit) cubes, say, and after a
tedious $\epsilon$-calculus a limiting process yields the result.
However, Theorem 1 elegantly overcomes these technical
difficulties.
\end{remark}

Furthermore, we also introduce a second notion of density as
follows.

\begin{definition}\label{finitedensity}
Let $G$ be a LCA group and $\mu:=\mu_G$ be its Haar measure. If
$\nu$ is another measure on $G$ with the sigma algebra of
measurable sets being ${\mathcal S}$, then we define
\begin{equation}\label{Fnudensity}
\overline{\D}(\nu;\mu) := \inf_{F\subset G,\,{\rm \#} F<\infty }
\sup_{V\in {\mathcal S} \cap \BB_0} \frac{\nu(V)}{\mu(F+V)}~.
\end{equation}
In particular, if $A\subset G$ is Borel measurable and $\nu=\mu_A$
is the trace of the Haar measure on the set $A$, then we get
\begin{equation}\label{FAdensity}
\overline{\D}(A) :=\overline{\Delta}(\nu_A;\mu) := \inf_{F\subset
G,\, {\rm \#} F<\infty } \sup_{V\in \BB_0} \frac{\mu(A\cap
V)}{\mu(F+V)}~.
\end{equation}
If $\Lambda\subset G$ is any (e.g. discrete) set and $\gamma
:=\gamma_\Lambda:=\sum_{\lambda\in\Lambda} \delta_{\lambda}$ is
the counting measure of $\Lambda$, then we get
\begin{equation}\label{FLdensity}
\overline{\D}^{\rm \#}(\Lambda)
:=\overline{\Delta}(\gamma_{\Lambda};\mu) := \inf_{F\subset G,\,
{\rm \# } F< \infty} \sup_{V\in \BB_0} \frac{{\rm \# }(\Lambda\cap
V)}{\mu(F+V)}~.
\end{equation}
\end{definition}

The two definitions are rather similar, except that the
requirements for $\overline{\D}$ refer to finite sets only.
Because all finite sets are necessarily compact in an LCA group,
\eqref{Cnudensity} of Definition \ref{compactdensity} extends the
same infimum over a wider family of sets than \eqref{Fnudensity}
of Definition \ref{finitedensity}; therefore we get

\begin{proposition}\label{prop:densitycompari}
Let $G$ be any LCA group, with normalized Haar measure $\mu$. Let
$\nu$ be any measure with sigma algebra of measurable sets
$\mathcal S$. Then we have
\begin{equation}\label{Fd-equivalance-gen}
\overline{\D}(\nu;\mu) \ge \overline{D}(\nu;\mu)~.
\end{equation}
\end{proposition}
This specializes to $\RR^d$ as follows.
\begin{proposition}\label{th:comparionRd}
Let us normalize the Haar measure of $\RR^d$ to be equal to the
volume $|\cdot|$. Let $\nu$ be any measure with sigma algebra of
measurable sets $\mathcal S$. Then we have
\begin{equation}\label{Fd-equivalance}
\overline{\D}(\nu;|\cdot|) \ge \overline{D}(\nu;|\cdot|)~.
\end{equation}
\end{proposition}
Moreover, the following is obvious, since in discrete groups the
Haar measure is the counting measure and the compact sets are
exactly the finite sets.
\begin{proposition}\label{prop:compactfinite} Let $\nu$ be any
measure on the sigma algebra $\SA$ of measurable sets in a
discrete Abelian group $G$. Take $\mu:=\#$ the counting measure,
which is the normalized Haar measure of $G$ as a LCA group. Then
\begin{equation}\label{finitecompact-equality}
\overline{\D}(\nu;\#) = \overline{D}(\nu;\#)~.
\end{equation}
\end{proposition}
So there is no difference for $\ZZ$, e.g. In general, however, the
two densities, defined above, may well be different: in fact, we
would bet for that, but we have no construction to show this.


\section{Proof of Theorem \ref{th:Requivalence}}
\label{sec:proof}

\noindent \underline{Proof of $\overline{D}(\nu;|\cdot|) \geq
\overline{D}_K(\nu)$}.

Let now $\tau<\tau'< \overline{D}_K(\nu)$ and $C\Subset G$ be
arbitrary. Since $C$ is compact, for some sufficiently large
$r'>0$ we have $C\Subset r'K$, hence by convexity also
$C+rK\subset (r+r')K$ for any $r>0$. On the other hand by $\tau'<
\overline{D}_K(\nu)$ there exist $r_n\to\infty$ and $x_n\in \RR^d$
with $|\nu(r_nK+x_n)|>\tau'|r_nK|$. With large enough $n$, we also
have $|(r_n+r')K|/|r_nK|=(1+r'/r_n)^d < \tau'/\tau$, hence with
$V:=r_n K+x_n$ we find $\nu(V) > \tau'|r_n K| > \tau |(r_n+r')K|=
\tau |x_n+r_n K +r'K| \geq \tau |V+C|$. This proves that
$\overline{D}(\nu;|\cdot|) \geq \tau$, whence the assertion.

\noindent \underline{Proof of $\overline{D}_K(\nu) \geq
\overline{D}(\nu;|\cdot|)$}.

Take now $\tau<\overline{D}(\nu;|\cdot|)$, put $C:=rK$ with some
$r>0$ given, and pick up a measurable set $V$ satisfying $\nu(V) >
\tau |V+C|$. We can then write
\begin{equation}\label{Vnu}
\int \chi_V(t) d\nu(t) > \tau |V+C|~.
\end{equation}
If $t\in V$, $u\in C(=rK)$, then $t+u\in V+C$, hence
$\chi_{V+C}(t+u)=1$, and we get
$$
\chi_V(t)\leq \frac1{|C|} \int \chi_{V+C}(t+u) \chi_C(u) du \qquad
(\forall t\in V)~.
$$
If $t\notin V$, this is obvious, as the left hand side vanishes:
hence \eqref{Vnu} implies
\begin{equation}\label{tauVC}
\tau|V+C| < \int \frac1{|C|} \int  \chi_{V+C}(t+u)\chi_C(u) du
d\nu(t) = \int \chi_{V+C}(y) \frac1{|C|} \int \chi_C(y-t)d\nu(t)
dy~.
\end{equation}
Sinve $C=-C$, the inner function is
$$
f(y):=\frac1{|C|} \int \chi_C(y-t)d\nu(t) = \frac{\nu(C+y)}{|C|}~,
$$
and according to \eqref{tauVC} we have $\tau |V+C| <  \int
\chi_{V+C}(y) f(y) dy = \int_{V+C} f$, hence for some appropriate
point $z\in V+C$ we must have $\tau < f(z)$. That is, $\nu(C+z) >
\tau |C|$, and we get by $C:=rK$ the estimate
\begin{equation}\label{festimate}
\nu(rK+z) > \tau |rK|.
\end{equation}
Since $r$ was arbitrary, it follows that $\overline{D}_K(\nu) \geq
\tau$, and applying this to all $\tau<\overline{D}(\nu;|\cdot|)$
the statement follows.


\section{Extension of the propositions
of Erd\H os-S\'ark\"ozy, of Hegyv\'ari, and of
F\"urstenberg\label{sec:additiveproposgeneralized}}

\begin{theorem}\label{th:general-hegyvari} If $G$ is a LCA group
and $A\subset G$ has $\overline{\D}(A)>0$, then there exists a
finite subset $B\subset G$ so that $A-A+B=G$. Moreover, we can
find $B$ with $\# B \le [1/\overline{\D}(A)]$.
\end{theorem}
\begin{remark} We need a translation-invariant (Haar) measure, but
not the topology or compactness.
\end{remark}
\begin{proof} Assume that $H\subset G$ satisfies
$(A-A)\cap(H-H)=\{0\}$ and let $L=\{b_1,b_2,\dots,b_k\}$ be any
finite subset of $H$. By condition, we have $(A+b_i)\cap
(A+b_j)=\emptyset$ for all $1\le i < j\le k$. Take now $C:=L$ in
the definition of density \eqref{FLdensity} and take
$0<\tau<\rho:=\overline{\D}(A)$. By Definition \ref{finitedensity}
of the density $\overline{\D}(A)$, there are $x\in G$ and $V
\subset G$ open with compact closure -- or, a $V\in \SA$ with
$0<|V|<\infty$ -- satisfying
\begin{equation}\label{AVx}
|A\cap(V+x)|>\tau|V+L|~.
\end{equation}
On the other hand
\begin{equation}\label{VLetc}
V+L=\bigcup_{j=1}^k \left(V+x+(b_j-x) \right)\supset
\bigcup_{j=1}^k \left( ((V+x)\cap A)+b_j \right)-x
\end{equation}
and as $A+b_j$ (thus also $((V+x)\cap A) +b_j$) are disjoint, and
the Haar measure is translation invariant, we are led to
\begin{equation}\label{kVxA}
|V+L|\ge k|(V+x)\cap A|~.
\end{equation}
Comparing \eqref{AVx} and \eqref{kVxA} we obtain
\begin{equation}\label{tauk}
|A\cap(V+x)|>\tau k|(V+x)\cap A|\qquad\qquad \text{and also}
\qquad\qquad |V+L| > k\tau |V+L|~,
\end{equation}
hence after cancellation by $|V+L|>0$ we get $k<1/\tau$ and so in
the limit $k\le K:=[1/\rho]$. It follows that $H$ is necessarily
finite and $\# H\le K$.

So let now $B=\{b_1,b_2,\dots,b_k\}$ be any set with the property
$(A-A)\cap(B-B)=\{0\}$ (which implies $\# B \le K$) and maximal in
the sense that for no $b'\in G\setminus B$ can this property be
kept for $B':=B\cup\{b'\}$. In other words, for any $b'\in
G\setminus B$ it holds that $(A-A)\cap (B'-B')\ne\{0\}$.

Clearly, if $A-A=G$ then any one point set $B:=\{b\}$ is such a
maximal set; and if $A-A\ne G$, then a greedy algorithm leads to
one in $\le K$ steps.

Now we can prove $A-A+B=G$. Indeed, if there exists $y\in
G\setminus(A-A+B)$, then $(y-b_j)\notin A-A$ for $j=1,\dots,k$,
hence $B':=B\cup\{y\}$ would be a set satisfying
$(B'-B')\cap(A-A)=\{0\}$, contradicting maximality of $B$.
\end{proof}

\begin{corollary} Let $A\subset \RR^d$ be a (measurable) set with
$\overline{\D}(A)>0$. Then there exists $b_1,\dots,b_k$ with $k
\le K := [1/\overline{\D}(A)]$ so that $\cup_{j=1}^k
(A-A+b_j)=\RR^d$.
\end{corollary}

This is interesting as it shows that the difference set of a set
of positive Banach density $\overline{\D}$ is necessarily rather
large: just a few translated copies cover the whole space.

Observe that we have Proposition \ref{prop:Füst} as an immediate
consequence, since $\ZZ$ is discrete, and thus the two notions
$\overline{\D}$ and $\overline{D}$ of Banach densities coincide;
moreover, the finite set $B:=\{b_1,\dots,b_K\}$ is a compact set
in the discrete topology of $\ZZ$. But in fact we can as well
formulate the following extension.

\begin{corollary}\label{cor:genFüst} Let $G$ be a LCA group and
$S\subset G$ a set with positive upper Banach density, i.e.
$\overline{D}(S)>0$, where here
$\overline{D}(S)=\overline{D}(\mu|_S;\mu)$. Then the difference
set $S-S$ is a syndetic set: moreover, the set of translations
$K$, for which we have $G=KS$, can be chosen not only compact, but
even to be a finite set with $\# K \leq [1/\overline{D}(S)]$
elements.
\end{corollary}

This corollary is immediate, because $\overline{\D}(S)\geq
\overline{D}(S)$ according to Proposition
\ref{prop:densitycompari}.

This indeed generalizes the proposition of F\"urstenberg. Also
this result contains the result of Hegyv\'ari: for on
$\sigma$-finite groups the natural topology is the discrete
topology, whence the natural Haar measure is the counting measure,
and so on $\sigma$-finite groups Corollary \ref{cor:genFüst} and
Theorem \ref{th:general-hegyvari} coincides. Finally, this also
generalizes and sharpens the Proposition of Erd\H os and
S\'ark\"ozy. Indeed, on $\ZZ$ or $\NN$ we naturally have
$\overline{\D}(A)=\overline{D}(A)\geq \overline{d}(A)$, so if the
latter is positive, then so is $\overline{D}(A)$; and then the
difference set is syndetic, with finitely many translates
belonging to a translation set $K$, say, covering the whole $\ZZ$.
Hence $d_{n+1}-d_n$ is necessarily smaller than the maximal
element of the finite set $K$ of translations.

\begin{theorem}\label{thm:strongFüsi} Let $G$ be a LCA group and
$S\subset G$ a set with a positive, (but finite) uniform
asymptotic upper density, regarding now the counting measure of
elements of $S$ in the definition of Banach density, i.e.
$\overline{D}(S)= \overline{D}(\#|_S;\mu)>0$. Then the difference
set $S-S$ is a syndetic set.
\end{theorem}
\begin{remark} One would like to say that a density $+\infty$ is
"even the better". However, in non-discrete groups this is not the
case: such a density can in fact be disastrous. Consider e.g. the
set of points $S:=\{1/n~:~n\in \NN\}$ as a subset of $\RR$.
Clearly for any compact $C$ of positive Haar /i.e. Lebesgue/
measure $|C|>0$, and for any $V\in \BB_0$ of finite measure and
compact closure, $|V+C|$ is positive but finite: whence whenever
$0\in {\rm int} V$, we automatically have $\#(S\cap V)=\infty$ and
also $\#(S\cap V)/|C+V|=\infty$, therefore
$\overline{D}(\#|_S;|\cdot|)=\infty$; but $S-S\subset [-2,2]$ and
thus with a compact $B$ it is not possible that $B+S-S$ covers
$G=\RR$, whence $S-S$ is not syndetic.
\end{remark}
\begin{problem}
The implicitly occurring set of translations $K$, for which we
have $G=K+(S-S)$, is not controlled in size by the proof below.
However, one would like to say that there must be some bound,
hopefully even $\mu (K) \leq [1/\overline{D}(S)]$, for an
appropriately chosen compact set of translates $K$. This we cannot
prove yet.
\end{problem}
\begin{proof} We are not certain that our argument is the simplest
possible: also, it does not give a good estimate for the measure
of the required compact set exhibiting the syndetic property of
$S-S$. Nevertheless, we consider it worthwhile to present it in
full detail, since the various steps, eventually leading to the
result, seem to be rather general and useful auxiliary statements,
having their own independent interest. Correspondingly, we break
the argument in a series of lemmas.

\begin{lemma}\label{l:packdensity} Let $S\subset G$ and assume
$\overline{D}(\#|_S;\mu) =\rho\in (0,\infty)$. Consider any
compact set $H \subset G$ satisfying the "packing type condition"
$H-H\cap S-S =\{0\}$ with $S$. Then we necessarily have $\mu(H)
\leq 1/\overline{D}(S)$.
\end{lemma}
\begin{proof}
Let $0<\tau<\rho$ be arbitrary. By definition of
$\overline{D}(S)$, (using $H$ in place of $C$) there must exist a
measurable set $V\in \SA\cap \BB_0$, with compact closure so that
$\infty>\#(S\cap V)>\tau \mu(V+H)$, therefore also $\#(S\cap
V)>\mu((S\cap V)+H)$. However, for any two elements $s\ne s' \in
(S\cap V)\subset S$, $(s+H)\cap (s'+H) =\emptyset$, since in case
$g\in (s+H)\cap (s'+H)$ we have $g=s+h=s'+h'$, i.e. $s-s'=h-h'$,
which is impossible for $s\ne s'$ and $(H-H)\cap (S-S)=\{0\}$.
Therefore for each $s\in (S\cap V)$ there is a translate of $H$,
totally disjoint  from all the others: i.e. the union $(S\cap
V)+H= \cup_{s\in(S\cap V)}(s+H)$ is a disjoint union. By the
properties of the Haar measure, we thus have $\mu((V\cap
S)+H)=\sum_{s\in(S\cap V)} \mu(s+H)= \#(V\cap S)\mu(H)$.

Whence we find $\#(S\cap V) \geq \tau \#(S\cap V) \mu(H)$, and,
since $\#(S\cap V)>\tau \mu(V+H)$ was positive, we can cancel with
it and infer $\mu(H)<1/\tau$. This holding for all
$\tau<\rho=\overline{D}(S)$, we obtained that any compact set $H$,
satisfying the packing type condition with $S$, is necessarily
bounded in measure by $1/\overline{D}(S)$.
\end{proof}

\begin{lemma}\label{l:fattening} Suppose that $S-S\cap H-H=\{0\}$
with $\overline{D} (\#|_S;\mu) =\rho\in (0,\infty)$ and $H\Subset
G$ with $0<\mu(H-H)$. Then the set $A:=S+(H-H)$ has the uniform
asymptotic upper density $\overline{D}(\mu|_A;\mu)$, with respect
to the Haar measure (restricted to $A$), not less than $\rho\cdot
\mu(H-H)$.
\end{lemma}
\begin{proof} Let $C\Subset G$ be arbitrary and denote $Q:=H-H$.
We want to estimate from below the ratio $\mu(A\cap V)/\mu(C+V)$
for an appropriately chosen $V\in \BB_0$. Let us fix that we will
take for $V$ some set of the form $U+Q$ with $U\in \BB_0$. Clearly
$A\cap V = (S+Q)\cap (U+Q) \supset (S\cap U)+Q$. Now for any two
elements $s\ne t \in S$, thus even more for $s,t \in (S\cap V)$,
the sets $s+Q$ and $t+Q$ are disjoint, this being an easy
consequence of the packing property because $s+q=t+q'
\Leftrightarrow s-t=q-r$, which is impossible for $s-t\ne 0$ by
condition. Therefore by the properties of the Haar measure we get
$\mu((S\cap U)+Q)=\sum_{s\in(S\cap U)} \mu(s+Q)=\#(S\cap U) \cdot
\mu(Q)$. In all, we found $\mu(A\cap V)\geq \#(S\cap U) \cdot
\mu(Q)$.

It remains to choose $V$, that is, $U$, appropriately. For the
compact set $C+Q\Subset G$ and for any given small $\ve>0$, by
definition of $\overline{D} (\#|_S;\mu)=\rho$ there exists some
$U\in \BB_0$ such that $\#(S\cap U) >(\rho-\ve) \mu((C+Q)+U)$.
Choosing this particular $U$ and combining the two inequalities we
are led to $\mu(A\cap V)\geq (\rho-\ve) \mu(C+Q+U) \mu(Q)$, that
is, for $V:=U+Q$ written in $\mu(A\cap V)/\mu(C+V)\geq (\rho-\ve)
\mu(H-H)$.

As we find such a $V$ for every positive $\ve$, the sup over
$V\in\BB_0$ is at least $\rho\mu(H-H)$, and because $C\Subset G$
was arbitrary, we infer the assertion.
\end{proof}

\begin{lemma}\label{l:ifclearthengood} Suppose that $S-S\cap
H-H=\{0\}$ with $\overline{D} (\#|_S;\mu) =\rho\in (0,\infty)$ and
$H\Subset G$ with $0<\mu(H-H)$. Then there exists a finite set
$B=\{b_1,\dots,b_k\}\subset G$ of at most $k\leq
[1/(\rho\mu(H-H))]$ elements so that $B+(H-H)-(H-H)+(S-S)=G$. In
particular, the set $S-S$ is syndetic with the compact set of
translates $B+(H-H+H-H)$.
\end{lemma}
\begin{proof} By the above Lemma \ref{l:fattening} we have
an estimate on the density of $A:=S+(H-H)$ with respect to Haar
measure. But then we may apply Corollary \ref{cor:genFüst} to see
that the difference set $S+(H-H)-(S+(H-H))$ is a syndetic set with
the set of translates $B$ admitting $\#B\leq
[1/\overline{D}(\mu|_A;\mu)]\leq [1/(\rho\mu(H-H))]$. Because also
the set $H$ is compact, this yields that $S$ is syndetic as well,
with set of translations being $B+(H-H)+(H-H)$.
\end{proof}

One may think that it is not difficult, for a discrete set $S$ of
finite density with respect to counting measure, to find a compact
neighborhood $R$ of 0, so that $R\cap (S-S)$ be almost empty with
0 being its only element. If so, then by continuity of
subtraction, also for some compact neighborhood $H$ of zero with
$(H-H)\subset R$ (and, being a neighborhood, with $\mu(H)>0$, too)
we would have $(H-H)\cap(S-S)=\{0\}$, the packing type condition,
whence concluding the proof of Theorem \ref{thm:strongFüsi}.

Unfortunately this idea turns to be naive. Consider the sequence
$S=\{n+1/n ~:~ n\in \NN\} \cup \NN$ (in $\RR$), which has uniform
asymptotic upper density 2 with the cardinality measure, whilst
$S-S$ is accumulating at 0.

Nevertheless, this example is instructive. What we will find, is
that sets of \emph{finite} positive uniform asymptotic upper
density cannot have a too dense difference set: it always splits
into a fixed, bounded number of disjoint subsets so that the
difference set of each subset already leaves out a fixed compact
neighborhood of 0. This will be the substitute for the above naive
approach to finish our proof of Theorem \ref{thm:strongFüsi}
through proving also some kind of subadditivity of the uniform
asymptotic upper density -- another auxiliary statement
interesting for its own right.

\begin{lemma}\label{l:partition} Let $Q\Subset G$ be any symmetric
compact neighborhood of 0 and let $S$ have positive but finite
uniform asymptotic upper density with respect to cardinality
measure, i.e. $\overline{D} (\#|_S;\mu) =\rho\in (0,\infty)$. Then
there exists a finite disjoint partition $S=\bigcup_{j=1}^n S_j$
of $S$ such that $(S_j-S_j)\cap Q =\{0\}$. Moreover, choosing an
appropriate symmetric compact neighborhood $Q$ of 0, depending on
$\ve>0$, we can even guarantee that the number of subsets in the
partition is not more than $k\leq [(1+\ve)\rho\mu(Q)]$.
\end{lemma}
\begin{proof} Let $s\in S$ be arbitrary, consider $R:=s+Q$, and
let us try to estimate the number of other elements of $S$ falling
in $R$. Clearly for any $C\Subset G$ we have $\#(S\cap
R)/\mu(C+R)\leq \sup_{V\in\BB_0}\#(S\cap V)/\mu(C+V)$ so for any
$\ve>0$ and with some appropriate $C\Subset G$ this is bounded by
$\rho+\ve$ according to the density condition. Note that the
choice of $C$ depends only on $\ve$, but not on $R$. That is, we
already have a bound $k:=\#(S\cap R)\leq (\rho+\ve) \mu(C+R)$ with
the given $C=C(\ve)$, independently of $R$, i.e. of $Q$.

Next we show how to obtain the bound $k\leq [\rho\mu(Q)]+1$ for
some appropriate choice of $Q$. This hinges upon a lemma of Rudin,
stating that for any given compact set $C\Subset G$ and $\ve>0$
there exists another Borel set $V$, also with compact closure, so
that $\mu(C+V)<(1+\ve)\mu(V)$, c.f. 2.6.7 Theorem on page 52 of
\cite{rudin:groups}; moreover, Rudin remarks that this can even be
proved (actually, read out from the proof) with open sets $V$
having compact closure. It is a matter of invariance of Haar
measure with respect to translations to ascertain that (some) of
the interior points of $V$ be 0, so that $V$ is a neighborhood of
0: also, by regularity of the Borel measure, and by compactness of
the closure, we can as well take $V$ to be its own closure.
Furthermore, the same proof also shows that $V$ can even be taken
symmetric. In all, \emph{for an appropriate choice of $V$ for
$Q$}, we even have $k:=\#(S\cap R)\leq (\rho+\ve) \mu(C+R)<
(\rho+\ve)(1+\ve) \mu(Q)$. Note that here the dependence on $C$
disappears from the end formula, but there is a dependence of $Q$
on $\ve$. This is equivalent to the estimate in the Lemma.

It remains to construct the partition once we have a compact
neighborhood $Q$ of 0 and a finite number $k\in \NN$ such that
$\#(S\cap(Q+s)\leq k$ for all $s\in S$. this is standard argument.
Consider a graph on the points of $S$ defined by connecting two
points $s$ and $t$ exactly when $t\in s+Q$. Since $Q$ is
symmetric, this is indeed a good definition for a graph (and not
for a directed graph only).

In this graph by condition the degree of any point of $s\in S$ is
at most $k-1$: there are at most $k-1$ further points of $S$ in
$s+Q$. But it is well-known that such a graph can be partitioned
into $k$ subgraphs with no edges within any of the induced
subgraphs.\footnote{The proof of this is very easy for finite or
countable graphs: just start to put the points, one by one,
inductively into $k$ preassigned sets $S_j$ so that each point is
put in a set where no neighbor of it stays; since each point has
less than $k$ neighbors, this simple greedy algorithm can not be
blocked and the points all find a place. Same for countable many
points, while for larger cardinalities transfinite induction is
needed to carry out the same reasoning.} That is, the set of
points split into the disjoint union of some $S_j$ with no two
points $s,t\in Q$ being in the relation $t\in s+Q$, defining an
edge between them.

It is easy to see that now we constructed the required partition:
the $S_j$ are disjoint, and so are $(S_j-S_j)$ and
$Q\setminus\{0\}$, for any $j=1,\dots,k$, too. This concludes the
proof.
\end{proof}

\begin{lemma}[subadditivity]\label{l:subadditivity} Let
$\nu_0=\sum_{j=1}^n \nu_j$ be a sum of measures, all on the common
set algebra $\SA$ of measurable sets. Then we have
$\overline{D}(\nu_0,\mu) \leq \sum_{j=1}^n
\overline{D}(\nu_j,\mu)$. In particular, this holds for one given
measure $\nu$ and a disjoint union of sets $A_0=\cup_{j=1}^n A_j$,
with $\nu_j:=\nu|_{A_j}$, for $j=0,1,\dots,k$.
\end{lemma}
\begin{proof} Uniform asymptotic upper density is clearly
monotone in the sets considered, therefore all $S_j$ have a
density $0\leq \rho_j\leq \rho<\infty$ Let $\ve>0$ be arbitrary,
and take $C_j\Subset G$ so that for all $V\in \BB_0$ in the
definition of $\overline{D}(\nu|_A;\mu)$ we have $\nu_j(V)\leq
(\rho_j+\ve)\mu(C_j+V)$. Such $C_j$ exists in view of the infinum
on $C\Subset G$ in the definition of u.a.u.d.

Consider the (still) compact set $C:=C_1+\dots+C_n$. By definition
of u.a.u.d. there is $V\in\BB_0$ such that $\nu(V)\geq (\rho-\ve)
\mu(C+V)$. Obviously, $\mu(C_j+V)\leq \mu(C+V)$, so on combining
these we obtain
$$
\rho-\ve \leq \frac{\nu(V)}{\mu(C+V)} = \frac{\sum_{j=1}^k
\nu_j(V)}{\mu(C+V)}  \leq \sum_{j=1}^k \frac{\nu_j(V)}{\mu(C_j+V)}
\leq \sum_{j=1}^k (\rho_j+\ve),
$$
that is, $\rho-\ve\leq \sum_j (\rho_j+\ve)$ holding for all $\ve$,
we find $\rho\leq \sum_j \rho_j$, as was to be proved.
\end{proof}

{\it Continuation of the proof of Theorem \ref{thm:strongFüsi}}.
We take now an \emph{arbitrary} compact neighborhood $H\Subset G$
of $0$, with of course $\mu(H)>0$, and also $Q:=H-H$ again with
$0<\mu(Q)<\infty$ and $Q$ a symmetric neighborhood of 0. By Lemma
\ref{l:partition} there exists a finite disjoint partition
$S=\cup_{j=1}^n S_j$ with $(S_j-S_j)\cap(H-H)=\{0\}$. By
subadditivity of u.a.u.d. (that is, Lemma \ref{l:subadditivity}
above), at least one of these $S_j$ must have positive u.a.u.d.
$\rho_j$ (with respect to the counting measure), namely of density
$0< \rho/n \leq \rho_j \leq \rho < \infty$, with
$\rho:=\overline{D}(\#|_S,\mu)$.

Selecting such an $S_j$, we can apply Lemma
\ref{l:ifclearthengood} to infer that already $S_j$ -- hence also
$S\supset S_j$ -- is syndetic.
\end{proof}

\noindent {\sc\small
Alfr\' ed R\' enyi Institute of Mathematics, \\
Hungarian Academy of Sciences, \\
1364 Budapest, Hungary}\\
E-mail: {\tt revesz@renyi.hu}

\end{document}